\documentclass{ifacconf}
\usepackage{amsmath, amssymb}
\usepackage{natbib}            
\usepackage{graphicx}    
\usepackage{enumerate}    
\usepackage{url}          

\newenvironment{proof}[1][Proof]{\begin{trivlist}
\item[\hskip \labelsep {\bfseries #1}]}{\end{trivlist}}
\newtheorem{lemma}{Lemma}[section]

\begin{document}

\begin{frontmatter}

\title{Boundary control of a singular reaction-diffusion equation on a disk} 


\author[First]{Rafael Vazquez} 
\author[Second]{Miroslav Krstic} 

\address[First]{Department of Aerospace Engineering, Universidad de Sevilla, Camino de los Descubrimiento s.n., 41092 Sevilla, Spain (e-mail: rvazquez1@us.es).}                                              
\address[Second]{Department of Mechanical and Aerospace Engineering, University of California San Diego, La Jolla, CA 92093-0411, USA (e-mail: krstic@ucsd.edu).}


\begin{abstract}                          
Recently, the problem of boundary stabilization for unstable linear constant-coefficient reaction-diffusion equation on $n$-balls (in particular, disks and spheres) has been solved by means of the backstepping method. However, the extension of this result to spatially-varying coefficients is far from trivial. As a first step, this work deals with radially-varying reaction coefficients under revolution symmetry conditions on a disk (the 2-D case). Under these conditions, the equations become singular in the radius. When applying the backstepping method, the same type of singularity appears in the backstepping kernel equations. Traditionally, well-posedness of the kernel equations is proved by transforming them into integral equations and then applying the method of successive approximations. In this case, the resulting integral equation is singular. A successive approximation series can still be formulated, however its convergence is challenging to show due to the singularities. The problem is solved by a rather non-standard proof that uses the properties of the Catalan numbers, a well-known sequence frequently appearing in combinatorial mathematics.
\end{abstract}

\end{frontmatter}


\section{Introduction}
In a series of previous results, the problem boundary stabilization for unstable linear constant-coefficient reaction-diffusion equation on $n$-balls has been solved. In particular~\cite{disk} and~\cite{disk2} describe, respectively, the (full-state) control design for the particular case of a 2-D disk and a general $n$-ball; that same design, augmented with an observer, is applied in~\cite{jie} to multi-agent deployment in 3-D space, with the agents distributed on a disk-shaped grid and commanded by leader agents located at the boundary. The output-feedback generalization to $n$-balls is presented in~\cite{nball}. Older, related results that use backstepping include the design an output feedback law for a convection problem on an annular domain (see~\cite{convloop}, also~\cite{chengkang-xie}), or observer design on cuboid domains~\cite{kugi}. However, going from an annular domain to a disk (which includes the origin) complicates the design, as singularities appear on the equations and have to be dealt with.

This work can be seen as a first step towards extending this family of previous results to the non-constant coefficient case, by assuming a certain symmetry for the initial conditions, which simplifies the problem. There have been specific results on disk- or spherical-shaped domains, such as~\cite{Prieur} and~\cite{scott}, which have assumed these same symmetry conditions.

Based on the domain shape we use polar coordinates, and using the symmetry of the initial conditions and imposing an equally symmetric controller, the system is transformed into a single 1-D system with  singular terms. We design a feedback law for this system using the backstepping method~\cite{krstic}.  The backstepping method has proved itself to be an ubiquitous method for PDE control, with many other applications including, among others, flow control~\cite{vazquez,vazquez-coron}, nonlinear PDEs~\cite{vazquez2}, hyperbolic 1-D systems~\cite{vazquez-nonlinear,florent,krstic3}, adaptive control~\cite{krstic4}, wave equations~\cite{krstic2}, and delays~\cite{krstic5}. The main idea of backstepping is finding an invertible transformation that maps the system into a stable \emph{target} system which needs to be chosen judiciously. To find the transformation, a hyperbolic partial differential equation (called the \emph{kernel equation}) needs to be solved. Typically, the well-posedness of the kernel equation is studied by transforming it into an integral equation and then applying successive approximations to construct a solution. The convergence of the successive approximation series guarantees that a solution always exists, it is unique, and it is bounded. However, in the problem posed in this paper, one obtains a \emph{singular} kernel equation. Following the previously-outlined procedure, one can transform it into a (singular) integral equation and then apply the method of successive approximations. However, proving the convergence of the resulting series is challenging. The main technical contribution of this paper is tackling this issue. We use a rather non-standard proof based on a combinatorial sequence of integers (the Catalan numbers).

The structure of the paper is as follows. In Section~\ref{sec-plant} we introduce the problem and state our main result. We explain our design method and  find the control kernel equation in Section~\ref{sec-design}. Next, we prove its well-posedness in Section~\ref{sect-wp}. We conclude the paper with some remarks in Section~\ref{sec-conclusions}.

\section{2-D Reaction-Diffusion System on a Disk}\label{sec-plant}
Consider the following reaction-diffusion system on a disk, written in polar coordinates $(r,\theta)$:
\begin{eqnarray}\label{eqn-u}
u_t&=&\frac{\epsilon}{r} \left( r u_r \right)_r+\frac{\epsilon}{r^2} u_{\theta \theta}+\lambda(r) u,
\end{eqnarray}
evolving in the disk ${\cal D}_R=\{(r,\theta):r \in[0,R], \theta\in[0,2\pi)\}$, for $t>0$, with  boundary conditions 
\begin{eqnarray}\label{eqn-bcu}
u(t,R,\theta)&=&U(t,\theta),
\end{eqnarray}
where $U(t)$ is the actuation (we assume we can control all the boundary). Note that the system will be unstable for large values of $\lambda$.  

Denote by $L^2({\cal D}_R)$ the space of $L^2$ functions on the disk defined as usual. For the case when $\lambda$ does not depend on $r$, the following result was shown in~\cite{disk}:
\begin{thm}\label{th-main}
Consider (\ref{eqn-u})--(\ref{eqn-bcu}) with constant $\lambda>0$, with initial conditions $u_0\in L^2({\cal D}_R)$ and the following (explicit) full-state feedback law for $U$:
\begin{eqnarray}
 U(t,\theta)&=&-{1\over 2\pi}\frac{\lambda}{\epsilon}\int_0^R \rho \frac{\mathrm{I}_1\left[\sqrt{\frac{\lambda}{\epsilon}(R^2-\rho^2)}\right]}{\sqrt{\frac{\lambda}{\epsilon}(R^2-\rho^2)}}  
 \nonumber \\ && \times
\int_{-\pi}^\pi  \frac{(R^2- \rho^2) u(t,\rho,\psi)}{R^2+\rho^2- 2R \rho \cos\left(\theta-\psi \right)} d\psi d\rho,\quad \label{control}
\end{eqnarray}
where $\mathrm{I}_1$ is the first-order modified Bessel function of the first kind. Then system (\ref{eqn-u})--(\ref{eqn-bcu}) has a unique $L^2({\cal D}_R)$ solution, and the equilibrium profile $u\equiv 0$ is exponentially stable in the $L^2({\cal D}_R)$ norm, i.e., there exists $c_1,c_2>0$ such that
\begin{equation}
\Vert u(t,\cdot) \Vert_{L^2({\cal D}_R)} \leq c_1 \mathrm{e}^{-c_2t} \Vert u_0 \Vert_{L^2({\cal D}_R)}.
\end{equation}
\end{thm}

Now we try to extend this result to the case when $\lambda$ is a function of $r$. As a first approach, consider that the initial conditions do not depend on the angle (have revolution symmetry) and fix $U$ as a constant (also not depending on the angle $\theta$). Then, by symmetry, there is no angular dependence and one can drop the $\theta$ derivative in (\ref{eqn-u})--(\ref{eqn-bcu}), finding the following 1-D problem:
\begin{eqnarray}\label{eqn-usym}
u_t&=&\frac{\epsilon}{r} \left( r u_r \right)_r+\lambda(r) u,
\end{eqnarray}
for $r\in[0,R)$, $t>0$, with  boundary conditions 
\begin{eqnarray}\label{eqn-bcusym}
u(t,R)&=&U(t),
\end{eqnarray}

Stabilization of (\ref{eqn-usym})--(\ref{eqn-bcusym}) is simpler than (\ref{eqn-u})--(\ref{eqn-bcu}) because the system is now 1-D, but still challenging due to the singular terms in (\ref{eqn-usym}). In addition, an explicit expression for the controller is not possible as in (\ref{control}) due to the spatially-varying $\lambda(r)$. Our main result is as follows.
\begin{thm}\label{th-mainsym}
Consider (\ref{eqn-usym})--(\ref{eqn-bcusym}) with constant $\lambda>0$, with initial conditions $u_0\in L^2(0,R)$ not depending on the angle  and the following full-state feedback law for $U$:
\begin{eqnarray}
 U(t)&=&\int_0^R K(R,\rho) u(t,\rho) d\rho,\quad \label{controlsymm}
\end{eqnarray}
where the kernel $K(r,\rho)$ is obtained as the solution of
\begin{eqnarray}
 K_{rr} +\frac{ K_{r}}{r}-K_{\rho\rho}+\frac{ K_{\rho}}{\rho}-\frac{K}{\rho^2}
=\frac{\lambda(\rho)}{\epsilon} K \label{eqn-K}
\end{eqnarray}
with boundary conditions
\begin{eqnarray}
K(r,0)&=&0,\\
K(r,r) 
&=& -\int_0^r \frac{\lambda(\rho)  }{2 \epsilon }d\rho,\label{eqn-Kbc2}
\end{eqnarray}
in the domain $\mathcal T=\left\{(r,\rho):0\leq \rho \leq r \leq R\right\}$.
 Then system (\ref{eqn-u})--(\ref{eqn-bcu}) has a unique $L^2(0,R)$ solution, and the equilibrium profile $u\equiv 0$ is exponentially stable in the $L^2$ norm, i.e., there exists $c_1,c_2>0$ such that
\begin{equation}
\Vert u(t,\cdot) \Vert_{L^2(0,R)} \leq c_1 \mathrm{e}^{-c_2t} \Vert u_0 \Vert_{L^2(0,R)}.
\end{equation}
\end{thm}

In the next sections we prove Theorem~\ref{th-mainsym}. First, Section~\ref{sec-design} applies the backstepping method, finding a feedback law whose kernel is the solution of the singular hyperbolic PDE (\ref{eqn-K})--(\ref{eqn-Kbc2}), and shows the Theorem, assuming that kernel equation is well-posed and has a bounded solution. Section~\ref{sect-wp} deals with the well-posedness of the kernel PDE.
\section{Control law design and stability result}\label{sec-design}
Due to lack of space, we sketch the details of how to obtain (\ref{eqn-K})--(\ref{eqn-Kbc2}) by borrowing the results from a previous, more general publication. 
\subsection{Target system and backstepping transformation}\label{sec-kernels}
In~\cite{disk} the problem was solved by posing a backstepping transformation from (\ref{eqn-u})--(\ref{eqn-bcu}) to the \emph{target} system 
\begin{eqnarray}\label{eqn-w}
w_t&=&\frac{\epsilon}{r} \left( r w_r \right)_r+\frac{\epsilon}{r^2} w_{\theta \theta},\\
w(t,R,\theta)&=&0,\label{eqn-bcw}
\end{eqnarray}
a well-posed and stable target system (one just needs to take the mean value of the target system in~\cite{disk}). The transformation had the form
\begin{equation}
w(t,r,\theta)=
u(t,r,\theta)
-\int_0^r \int_{-\pi}^\pi  K(r,\rho,\theta,\psi) u(t,\rho,\psi) d\psi d\rho.\label{eqn-2dtran}
\end{equation}

Proceeding analogously, we pose a transformation without angular dependences
\begin{equation}
w(t,r)=
u(t,r)
-\int_0^r   K(r,\rho) u(t,\rho)  d\rho,\label{eqn-2dtransym}
\end{equation}
to reach the target system
\begin{eqnarray}\label{eqn-wsym}
w_t&=&\frac{\epsilon}{r} \left( r w_r \right)_r,\\
w(t,R)&=&0.\label{eqn-bcwsym}
\end{eqnarray}

The system (\ref{eqn-wsym})--(\ref{eqn-bcwsym}) inherits its stability properties from (\ref{eqn-w})--(\ref{eqn-bcw}), because it is a particular case of it (with initial conditions having revolution symmetry). Similarly, the transformation (\ref{eqn-2dtransym}) is a particular case of (\ref{eqn-2dtran}) and the resulting kernel equations can be directly extracted from~\cite{disk}. To do this, it must be noted that~\cite{disk} decomposed the system equation (\ref{eqn-u})--(\ref{eqn-bcu}) and the transformation  (\ref{eqn-2dtran}) in its Fourier components, and that having revolution symmetry is equivalent to only considering the mean Fourier component ($n=0$). Thus, the equation that $K(r,\rho)$ in (\ref{eqn-2dtransym}) has to verify to map (\ref{eqn-usym})--(\ref{eqn-bcusym})  into (\ref{eqn-wsym})--(\ref{eqn-bcwsym})  is directly obtained from~\cite{disk} as (\ref{eqn-K})--(\ref{eqn-Kbc2}).

Finally, applying the transformation (\ref{eqn-2dtransym}) at $r=R$ and using the boundary conditions of both target and original systems, (\ref{eqn-bcwsym}), (\ref{eqn-bcusym}) respectively, we obtain the feedback law (\ref{controlsymm}).
\subsection{Invertibility of the transformation}
It can be shown that if we pose an inverse transformation of the form
\begin{equation}
u(t,r)=
w(t,r)
+\int_0^r L(r,\rho) u(t,\rho,\psi)  d\rho,\label{eqn-2dinvtransym}
\end{equation}
we can find (using the same procedure of Section~\ref{sec-kernels}) that the inverse kernel $L$ verifies the following hyperbolic PDE
\begin{equation}
L_{rr} +\frac{ L_{r}}{r}-L_{\rho\rho}+\frac{ L_{\rho}}{\rho}-\frac{L}{\rho^2}
=-\frac{\lambda(r)}{\epsilon} L.\label{eqn-Ln}
\end{equation} 
with boundary conditions
\begin{eqnarray}
L(r,0)&=&0,\\
L(r,r) 
&=&-\int_0^r \frac{\lambda(\rho)  }{2 \epsilon }d\rho. \label{eqn-lnbc}
\end{eqnarray}
These equations are very similar to (\ref{eqn-K})--(\ref{eqn-Kbc2}) with slight differences. The proof that we will show in Section~\ref{sect-wp} can be applied to show there exists a bounded solution to (\ref{eqn-Ln})--(\ref{eqn-lnbc}).

Once it has been established that there are bounded direct and inverse backstepping transformations, then it is easy to show that both transformations map $L^2$ functions into $L^2$ functions (see e.g.~\cite{nball}). Then, the well-posedness and stability properties of the target system (\ref{eqn-wsym})--(\ref{eqn-bcwsym}) are \emph{mapped} to the original system (\ref{eqn-usym})--(\ref{eqn-bcusym}), proving Theorem~\ref{th-mainsym}. It only remains to show that the kernel equations have a bounded solution, which is done next  in Section~\ref{sect-wp}.
\section{Well-posedness of the kernel PDE}\label{sect-wp}
Next we show the well-posedness of (\ref{eqn-K})--(\ref{eqn-Kbc2}), showing in particular that there exists a bounded solution by using a constructive method. This result is the main technical development of this paper which is necessary to complete the proof of Theorem~\ref{th-mainsym}.
\subsection{Transforming the kernel PDE into an integral equation}
To better analyze the kernel equation, define $G=\sqrt{\frac{r}{\rho}}K$. This is an allowed transformation given that $K$ is assumed to be differentiable and zero at $\rho=0$. Thus $K$, when close to $\rho=0$, behaves like $\rho$ and therefore it can be divided by $\sqrt{\rho}$. 

The equation verified by $G$ is
\begin{eqnarray}
 G_{rr} -G_{\rho\rho}+\frac{G}{4r^2}-\frac{G}{4\rho^2}
=\frac{\lambda(\rho)}{\epsilon} G
\end{eqnarray}
with boundary conditions
\begin{eqnarray}
G(r,0)&=&0,\\
G(r,r) 
&=& -\int_0^r \frac{\lambda(\rho)  }{2 \epsilon }d\rho.
\end{eqnarray}

Following~\cite{krstic}, define $\alpha=r+\rho$, $\beta=r-\rho$. Then, the $G$ equation in $\alpha,\beta$ variables becomes
\begin{eqnarray}
4 G_{\alpha\beta}+\frac{G}{(\alpha+\beta)^2}-\frac{G}{(\alpha-\beta)^2} =\frac{\lambda\left(\frac{\alpha-\beta}{2}\right)}{\epsilon} G
\end{eqnarray}
in the domain $\mathcal T'=\left\{(\alpha,\beta):0\leq \beta \leq \alpha \leq 2R, \beta \leq 2R-\alpha \right\}$ with boundary conditions
\begin{eqnarray}
G(\beta,\beta)&=&0,\\
G(\alpha,0) 
&=& -\int_0^{\alpha/2} \frac{\lambda(\rho)  }{2 \epsilon }d\rho.
\end{eqnarray}

This can be transformed into an integral equation as typical in backstepping~\cite{krstic}. First, we find
\begin{eqnarray}
 G_{\alpha\beta} =\frac{\lambda\left(\frac{\alpha-\beta}{2}\right)}{4\epsilon} G+\frac{\alpha\beta}{(\alpha^2-\beta^2)^2}G
 \end{eqnarray}
 now, integrating in $\beta$ from $0$ to $\beta$:
 \begin{eqnarray}
 G_{\alpha}(\alpha,\beta)- G_{\alpha}(\alpha,0)&=&\int_0^\beta \frac{\lambda\left(\frac{\alpha-\sigma}{2}\right)}{4\epsilon} G(\alpha,\sigma) d\sigma \nonumber
 \\ &&
 +\int_0^\beta\frac{\alpha\sigma}{(\alpha^2-\sigma^2)^2}G(\alpha,\sigma)d\sigma.\quad
 \end{eqnarray}
 and integrating again in $\alpha$ from $\beta$ to $\alpha$:
  \begin{eqnarray}
&& G(\alpha,\beta)-G(\beta,\beta)- G(\alpha,0)+G(\beta,0) 
 \nonumber \\
 &=&\int_\beta^\alpha \int_0^\beta \frac{\lambda\left(\frac{\eta-\sigma}{2}\right)}{4\epsilon} G(\eta,\sigma) d\sigma d\eta
 \nonumber \\ &&
 +\int_\beta^\alpha \int_0^\beta\frac{\eta\sigma}{(\eta^2-\sigma^2)^2}G(\eta,\sigma)d\sigma d\eta.
 \end{eqnarray}
 Using the boundary conditions:
   \begin{eqnarray}
 G(\alpha,\beta)&=&-\int_{\beta/2}^{\alpha/2} \frac{\lambda(\rho)  }{2 \epsilon }d\rho+\int_\beta^\alpha \int_0^\beta \frac{\lambda\left(\frac{\eta-\sigma}{2}\right)}{4\epsilon} G(\eta,\sigma) d\sigma d\eta
 \nonumber \\ &&
 +\int_\beta^\alpha \int_0^\beta\frac{\eta\sigma}{(\eta^2-\sigma^2)^2}G(\eta,\sigma)d\sigma d\eta.\label{eqn-inteq}
 \end{eqnarray}
 This is a singular integral equation due to the terms in the last integral. 
 
 \subsection{Successive approximations series}
 The method of successive approximations applied in~\cite{krstic} and posterior works to show that (\ref{eqn-inteq}) has a solution can be applied.  Thus, define
    \begin{eqnarray}
 G_0(\alpha,\beta)&=&-\int_{\beta/2}^{\alpha/2} \frac{\lambda(\rho)  }{2 \epsilon }d\rho \end{eqnarray}
and for $k>0$,
   \begin{eqnarray}
 G_k(\alpha,\beta)&=&\int_\beta^\alpha \int_0^\beta \frac{\lambda\left(\frac{\eta-\sigma}{2}\right)}{4\epsilon} G_{k-1}(\eta,\sigma) d\sigma d\eta \nonumber \\Ê&&
 +\int_\beta^\alpha \int_0^\beta\frac{\eta\sigma}{(\eta^2-\sigma^2)^2}G_{k-1}(\eta,\sigma)d\sigma d\eta.
 \end{eqnarray}
Then, the solution to the integral equation is
 \begin{equation}
 G=\sum_{k=0}^\infty G_k(\alpha,\beta),
 \end{equation}
 assuming the series converges. 
 
 However, proving convergence of the series is harder than usual. The typical procedure (see~\cite{krstic} and posterior works) is to assume a functional bound for $G_k$ and show by recursion it is verified for every $k$. In this case we follow a different method.
 
 Call $\bar \lambda=\max_{(\alpha,\beta)\in \mathcal T'} \left|  \dfrac{\lambda\left(\frac{\alpha-\beta}{2}\right)}{4\epsilon} \right|$.
 Then one clearly obtains $\vert G_0(\alpha,\beta)\vert \leq \bar \lambda (\alpha-\beta)$. However when trying to substitute this bound in the expression of $G_1$ we find an integral that is not so easy to compute. Instead, we formulate a series of technical results that will help deriving a functional bound for $G_k$.

\begin{lemma}\label{lem-Fnk}
Define, for $n\geq 0,k\geq 0$,
\begin{equation}
F_{nk}(\alpha,\beta)=\frac{\bar\lambda^{n+1}\alpha^n\beta^n}{n!(n+1)!}(\alpha-\beta)\frac{\log^k\left(\frac{\alpha+\beta}{\alpha-\beta}\right)}{k!}
.
\end{equation}
and $F_{nk}=0$ if $n<0$ or $k<0$.
 Then:
 \begin{enumerate}[(a)]
\item  $F_{nk}$ is well-defined and nonnegative in $\mathcal T'$ for all $n,k$
\item $F_{nk}(\beta,\beta)=0$ for all $n$ and $k$
\item $F_{nk}(\alpha,0)=0$ if $n\geq1$ or $k\geq 1$ 
\item $F_{00}(\alpha,0)=\alpha$
\item The following identity  is valid for $n\geq 1$ or $k\geq 1$.
\end{enumerate}
\begin{eqnarray}
F_{nk}&=&4\int_\beta^\alpha \int_0^\beta\frac{\eta\sigma\left(F_{n(k-1)}(\eta,\sigma) -F_{n(k-2)}(\eta,\sigma) \right)}{(\eta^2-\sigma^2)^2}  d\sigma d\eta \nonumber \\Ê&&
 +\int_\beta^\alpha \int_0^\beta \bar \lambda F_{(n-1)k}(\eta,\sigma) d\sigma d\eta.\qquad
\end{eqnarray}
\end{lemma}
\begin{proof}
First, it is easy to see that since $\frac{\alpha+\beta}{\alpha-\beta}>1$, the logarithm is always nonnegative. Also, for any $k$, $(\alpha-\beta)\log^k\left(\frac{\alpha+\beta}{\alpha-\beta}\right)$ is bounded since a linear term always dominates a logarithm (no matter the exponent of the logarithm). Thus $F_{nk}$ is well-defined in $\mathcal T'$. The values at the boundaries of $\mathcal T'$ are found by simple substitution.
Finally, for the integral, we have the following identity which is found by differentiation:
\begin{equation}
\frac{\partial^2F_{nk}(\alpha,\beta)}{\partial \alpha \partial \beta} =F_{(n-1)k}+4\frac{\alpha\beta\left(F_{n(k-1)}-F_{n(k-2)}\right)}{(\alpha^2-\beta^2)^2}.
\end{equation}
Thus
\begin{eqnarray}
&& \int_\beta^\alpha \int_0^\beta
\frac{\partial^2}{\partial \alpha \partial \beta} F_{nk}(\eta,\sigma) d\sigma d\eta
\nonumber \\ 
&=&F_{nk}(\alpha,\beta)-F_{nk}(\beta,\beta)-F_{nk}(\alpha,0)+F_{nk}(\beta,0)
\nonumber \\&=&4\int_\beta^\alpha \int_0^\beta\frac{\eta\sigma\left(F_{n(k-1)}(\eta,\sigma) -F_{n(k-2)}(\eta,\sigma) \right)}{(\eta^2-\sigma^2)^2}  d\sigma d\eta \nonumber \\Ê&&
 +\int_\beta^\alpha \int_0^\beta \bar \lambda F_{(n-1)k}(\eta,\sigma) d\sigma d\eta,
\end{eqnarray}
and solving for $F_{nk}(\alpha,\beta)$ and substituting the values at the boundaries, we reach the formula.\qed
\end{proof}

From the previous lemma, the following result is straightforward.
\begin{lemma}\label{lem-form}
For $n\geq 1$ or $k\geq 1$
\begin{eqnarray}
\sum_{i=1}^{i=k} F_{ni}&=&\int_\beta^\alpha \int_0^\beta \bar \lambda \sum_{i=1}^{i=k} F_{(n-1)i}(\eta,\sigma) d\sigma d\eta
\nonumber \\Ê&&
 +4\int_\beta^\alpha \int_0^\beta\frac{\eta\sigma}{(\eta^2-\sigma^2)^2} F_{n(k-1)}(\eta,\sigma)  d\sigma d\eta.\quad
\end{eqnarray}
\end{lemma}

We can use the formulas from Lemma~\ref{lem-form} to find the values of some bounds for $G_k$. For illustration, let us find the first values. Obviously
\begin{equation}
\vert G_0 \vert \leq F_{00}.
\end{equation}
Then
\begin{eqnarray}
\vert G_1 \vert &\leq& \int_\beta^\alpha \int_0^\beta \bar \lambda F_{00} (\eta,\sigma) d\sigma d\eta
\nonumber \\ &&
 +\int_\beta^\alpha \int_0^\beta\frac{\eta\sigma}{(\eta^2-\sigma^2)^2}F_{00}(\eta,\sigma)d\sigma d\eta
 \nonumber \\ &=&
F_{10}+\frac{F_{01}}{4}
\end{eqnarray}
where we have used the formulas of Lemma~\ref{lem-Fnk}. The next term is
\begin{eqnarray}
\vert G_2 \vert &\leq& \int_\beta^\alpha \int_0^\beta \bar \lambda\left(F_{10}
+\frac{F_{01}}{4}\right) d\sigma d\eta
\nonumber \\Ê&&
 +\int_\beta^\alpha \int_0^\beta\frac{\eta\sigma}{(\eta^2-\sigma^2)^2}\left(F_{10}+\frac{F_{01}}{4}\right)d\sigma d\eta
 \nonumber \\
 &=&F_{20}+\frac{F_{11}}{4}+\frac{F_{01}+F_{02}}{16}.
\end{eqnarray}
\begin{table*}[htp]
\caption{Catalan's Triangle}\label{tab-catalan}
\begin{center}
\begin{tabular}{|c|c|c|c|c|c|c|c|c|c|c|c|}
\hline
$C_{ij}$&$j=1$ &$j=2$ &$j=3$ &$j=4$ &$j=5$ &$j=6$&$j=7$ &$j=8$ &$j=9$ & $j=10$ \\ \hline
$i=1$ &1& &&&&&&& &\\\hline
$i=2$ &1&1 &&&&&&& &\\\hline
$i=3$ &2& 2&1&&&&&& &\\\hline
$i=4$ &5& 5&3&1 &&&&& &\\\hline
$i=5$ &14 &   14 & 9  &  4&1 &&&& &\\\hline
$i=6$ &  42  &      42  &    28 &   14 & 5 & 1&&& &\\\hline
$i=7$ & 132 &132&90&48&20&6&1&& &\\\hline
$i=8$ & 429&429&297&165&75&27&7&1& &\\\hline
$i=9$ & 1430 &1430&1001&572&275&110&35&8& 1&\\\hline
$i=10$ &4862 &4862&3432&2002&1001&429&154&44&9 &1\\\hline
\end{tabular}
\end{center}
\vspace{0.2cm}
\mbox{}
\end{table*}%
Similarly we find
\begin{equation}\label{eqn-g3}
\vert G_3 \vert < 
F_{30}+\frac{F_{21}}{4}+\frac{F_{11}+F_{12}}{16}+\frac{2F_{01}+2F_{02}+F_{03}}{64}
\end{equation}
Some particular number appear in these expressions. To try to identify the pattern, call:
\begin{eqnarray} 
H_1[F]&=& \int_\beta^\alpha \int_0^\beta \bar \lambda F (\eta,\sigma) d\sigma d\eta \nonumber \\
H_2[F]&=& \int_\beta^\alpha \int_0^\beta  \frac{\eta\sigma}{(\eta^2-\sigma^2)^2} F (\eta,\sigma) d\sigma d\eta,
\end{eqnarray}
which are the two operations to find a successive approximation term from the previous one. To find a bound on $G_4$ we have to apply $H_1$ and $H_2$ to (\ref{eqn-g3}) and use Lemmas~\ref{lem-Fnk} and~\ref{lem-form}, finding:
\begin{eqnarray}
&& H_1[F_{30}]=F_{40}\nonumber\\
&& H_2[F_{30}]+\frac{H_1[F_{21}]}{4}=\frac{F_{31}}{4}\nonumber\\
&& \frac{H_2[F_{21}]}{4}+\frac{H_1[F_{11}+F_{12}]}{16}=\frac{F_{21}+F_{22}}{16}\nonumber\\
&& \frac{H_2[F_{11}+F_{12}]}{16}+\frac{H_1[2F_{01}+2F_{02}+F_{03}]}{64}
\nonumber\\ &&
=\frac{2F_{11}+2F_{12}+F_{13}}{64}\nonumber\\
&& \frac{H_2[2F_{01}+2F_{02}+F_{03}]}{64}=\frac{5F_{01}+5F_{02}+3F_{03}+F_{04}}{256}.\qquad
\end{eqnarray}
Thus, we obtain
\begin{eqnarray}
\vert G_4 \vert &\leq& F_{40} + \frac{F_{31}}{4} +\frac{F_{21}+F_{22}}{16} +\frac{2F_{11}+2F_{12}+F_{13}}{64}+\nonumber \\Ê&&\frac{5F_{01}+5F_{02}+3F_{03}+F_{04}}{256}.
\end{eqnarray}

By extending this structure to the general case, we find the following recursive formula for $n>0$, expressed as a lemma.
\begin{lemma}\label{lem-gnrecursive}
For $n>0$, it holds that
\begin{eqnarray}\label{eqn-gnrecursive}
\vert G_n \vert \leq F_{n0}+\sum_{i=0}^{n-1}\sum_{j=1}^{j=n-i} \frac{C_{(n-i)j}}{4^{n-i}}F_{ij},
\end{eqnarray}
where the numbers $C_{ij}$ verify 
\begin{enumerate}[(a)]
\item  $C_{11}=1$
\item $C_{i0}=0$
\item $C_{ij}=0$ if $j>i$, for all $i$.
\item $C_{ij}=C_{(i-1)(j-1)}+C_{i(j+1)}$ for all other values of $i$ and $j$.
\end{enumerate}
\end{lemma}

The set of numbers in Lemma~\ref{lem-gnrecursive}, known as the ``Catalan's Triangle'', or the ballot numbers (see. e.g.~\cite{sloane} and references therein, even though the numbers are written in a slightly different ordering). The first few numbers are shown in Table~\ref{tab-catalan}. In particular, the first column of Table~\ref{tab-catalan}, this is, what we have called $C_{i1}$, are the Catalan numbers as they are usually defined~\cite[p.265]{brualdi}. Both the Catalan numbers and Catalan's Triangle verify many interesting properties, and are connected to a wide set of combinatorial problems as well as other number sets, such as the coefficients of certain Chebyshev polynomials (see e.g.~\cite[p. 797]{abramowitz}, where the nonnegative rows of table 22.8 are the columns of Table~\ref{tab-catalan}).

Let us establish some properties about these numbers before proving Lemma~\ref{lem-gnrecursive}.
\begin{lemma}
For $0<j \leq i$ it holds that
\begin{enumerate}[(a)]
\item $C_{ii}=1$.
\item $C_{ij}=\sum_{k=j-1}^{i-1}C_{(i-1)k}$.
\end{enumerate}
\end{lemma}
\begin{proof}
We show (a) by induction on $i$. For $i=1$, $C_{11}=1$ by definition. Assuming it true for $i$, then $C_{(i+1)(i+1)}=C_{ii}+C_{(i+1)(i+2)}=C_{ii}=1$.

We show (b) by descending induction on $1\leq j\leq i$. For $j=i$, we obviously have $C_{ii}=\sum_{k=i-1}^{i-1}C_{(i-1)k}=C_{(i-1)(i-1)}=1$, as we just showed. Assuming it true for $j$, we show it for $j-1$ by using the definition of the numbers:
\begin{eqnarray}
C_{i(j-1)}&=&C_{(i-1)(j-2)}+C_{ij}
\nonumber \\ 
&=&C_{(i-1)(j-2)}+\sum_{k=j-1}^{i-1}C_{(i-1)k}
\nonumber \\ 
&=&\sum_{k=j-2}^{i-1}C_{(i-1)k},
\end{eqnarray}
thus proving the Lemma.\qed
\end{proof}

\begin{proof}[{Proof of Lemma~\ref{lem-gnrecursive}}]
We establish the lemma by induction on $n\geq 2$. For $G_2$ it is already established. Now assume it is true for $G_{n}$ and prove it for $G_{n+1}$. We have:
\begin{eqnarray}
\vert G_{n+1} \vert \leq H_1[\vert G_{n}\vert]+H_2[\vert G_{n}\vert]
\end{eqnarray}
Thus
\begin{eqnarray}
\vert G_{n+1} \vert &\leq& H_1\left[F_{n0}+\sum_{i=0}^{n-1}\sum_{j=1}^{j=n-i} \frac{C_{(n-i)j}}{4^{n-i}}F_{ij}\right]
\nonumber \\ &&
+H_2\left[F_{n0}+\sum_{i=0}^{n-1}\sum_{j=1}^{j=n-i} \frac{C_{(n-i)j}}{4^{n-i}}F_{ij}\right],
%
\end{eqnarray}
and since the integral $H_1$, $H_2$ operators are linear, we can express this inequality in a convenient way as follows
\begin{eqnarray}
\vert G_{n+1} \vert &\leq&  H_1\left[F_{n0}\right] +\left(H_2\left[F_{n0}\right]+H_1\left[\frac{C_{11}F_{(n-1)1}}{4}\right]\right)
\nonumber \\ &&
+\sum_{i=1}^{n-1}
\left( H_2\left[\sum_{j=1}^{j=n-i} \frac{C_{(n-i)j}}{4^{n-i}}F_{ij}\right]
\right. \nonumber \\ && \left.
+H_1\left[
\sum_{j=1}^{j=n-i+1} \frac{C_{(n-i+1)j}}{4^{n-i+1}}F_{{i-1}j}
\right]
\right)
\nonumber \\ &&
+H_2\left[\sum_{j=1}^{j=n} \frac{C_{nj}}{4^{n}}F_{0j}\right]. \label{eqn-expression}
%
\end{eqnarray}
We next manipulate some of the lines of (\ref{eqn-expression}) to reach the final result. First, for the second line of (\ref{eqn-expression}), we have that
\begin{eqnarray}
&&H_2\left[\sum_{j=1}^{j=n-i} \frac{C_{(n-i)j}}{4^{n-i}}F_{ij}\right]\nonumber \\
&=&\frac{1}{4^{n-i}} \sum_{j=1}^{j=n-i} C_{(n-i)j}H_2[F_{ij}]
\nonumber \\
&=&\frac{1}{4^{n-i}} \sum_{j=1}^{j=n-i}C_{(n-i)j}\left(\frac{\sum_{l=1}^{j+1}\left( F_{il}-H_1[F_{(i-1)l}]\right)}{4}\right)
\nonumber \\
&=&\frac{1}{4^{n+1-i}} \sum_{j=1}^{j=n-i}\sum_{l=1}^{j+1} C_{(n-i)j} \left(F_{il}-H_1[F_{(i-1)l}]\right)
\nonumber \\
&=&\frac{1}{4^{n+1-i}}
 \sum_{l=1}^{l=n-i+1}
 \left(F_{il}-H_1[F_{(i-1)l}]\right)
 \sum_{j=l-1}^{n-i} 
C_{(n-i)j} 
\nonumber \\
&=&\frac{1}{4^{n+1-i}}
 \sum_{l=1}^{l=n-i+1}
 \left(F_{il}-H_1[F_{(i-1)l}]\right)
C_{(n-i+1)l}.\label{eqn-exp2}
\end{eqnarray}
and we can see that the sum of the second term in the parenthesis of (\ref{eqn-exp2}) cancels the third line of (\ref{eqn-expression}).

For the last line of (\ref{eqn-expression}), we have that
\begin{eqnarray}
H_2\left[\sum_{j=1}^{j=n} \frac{C_{nj}}{4^{n}}F_{0j}\right]
&=&\sum_{j=1}^{j=n} \frac{C_{nj}}{4^{n}}H_2\left[F_{0j}\right]\nonumber \\
&=&\sum_{j=1}^{j=n} \frac{C_{nj}}{4^{n}} \sum^{j+1}_{l=1} \frac{F_{0l}}{4} \nonumber \\
&=&\frac{1}{4^{n+1}} \sum_{j=1}^{j=n}\sum^{j+1}_{l=1} C_{nj}F_{0l} \nonumber \\
&=&\frac{1}{4^{n+1}} \sum_{l=1}^{l=n+1}F_{0l} \sum^{j=n}_{j=l-1} C_{nj} \nonumber \\
&=&\frac{1}{4^{n+1}} \sum_{l=1}^{l=n+1}C_{(n+1) l} F_{0l}. 
\end{eqnarray}

Thus we find that 
\begin{eqnarray}
\vert G_{n+1} \vert &\leq&
F_{(n+1)0}+\sum_{i=1}^{n-1}
\left( 
 \sum_{j=1}^{j=n-i+1}
\frac{C_{(n-i+1)j}}{4^{n+1-i}} F_{ij}
\right)
\nonumber \\ &&+\frac{C_{11}F_{n1}}{4}
+\frac{1}{4^{n+1}} \sum_{l=1}^{j=n+1}C_{(n+1) l} F_{0l} 
\nonumber \\
&=&F_{(n+1)0}+\sum_{i=0}^{n}\sum_{j=1}^{j=n+1-i} \frac{C_{(n+1-i)j}}{4^{n+1-i}}F_{ij}, \qquad
\end{eqnarray}
proving the Lemma.\qed
\end{proof}

Therefore, since the solution of the successive approximations series verifies
 \begin{equation}
 \vert G\vert \leq \sum_{n=0}^\infty \vert G_n(\alpha,\beta)\vert
 \end{equation}
  if we can prove the convergence of the series with the estimates of Lemma~\ref{lem-gnrecursive} that we just derived, we can prove the existence of a solution to the integral equation and therefore to the kernel equation. 
 
 \subsection{Convergence of the successive approximation series}
 By Lemma~\ref{lem-gnrecursive}, we find that the series whose convergence we need to study can be written as
 \begin{equation}\label{eqn-Gsum}
 \vert G\vert \leq \sum_{n=0}^\infty  F_{n0}+ \sum_{n=1}^\infty \sum_{i=0}^{n-1}\sum_{j=1}^{j=n-i} \frac{C_{(n-i)j}}{4^{n-i}}F_{ij}.
 \end{equation}
 
 The first term is easy to compute (see e.g.~\cite[p.375]{abramowitz}):
  \begin{eqnarray}
  \sum_{n=0}^\infty  F_{n0}&=&
    \sum_{n=0}^\infty 
  \frac{\bar\lambda^{n+1}\alpha^n\beta^n}{n!(n+1)!}(\alpha-\beta)
  \nonumber \\ 
& =&\sqrt{\bar \lambda}(\alpha-\beta) \frac{\mathrm{I_1}\left[2\sqrt{\bar \lambda \alpha \beta}\right]}{\sqrt{\alpha \beta}},\label{eqn-besselsum}
     \end{eqnarray}
     where $I_1$ is the first-order modified Bessel function of the first kind. For the next term, we use the fact that
  \begin{equation}
  \sum_{n=1}^\infty \sum_{i=0}^{n-1} H(n,i)
  =\sum_{i=0}^{\infty} \sum_{l=1}^\infty  H(l+i,i),
  \end{equation}
  therefore
   \begin{eqnarray}
&&\sum_{n=1}^\infty \sum_{i=0}^{n-1}\sum_{j=1}^{j=n-i} \frac{C_{(n-i)j}}{4^{n-i}}F_{ij}
\nonumber \\
&=& \sum_{i=0}^{\infty} \sum_{l=1}^\infty  \sum_{j=1}^{j=l}   \frac{C_{lj}}{4^{l}}F_{ij}\nonumber \\
&=& \sum_{i=0}^{\infty}  \sum_{j=1}^{j=\infty}\left(  \sum_{l=j}^\infty   \frac{C_{lj}}{4^{l}}\right)F_{ij}.
 \end{eqnarray}
It turns out that the parenthesis can be calculated and gives an exact sum for each $j$.

For that, we need only the fact (see any combinatorics book, e.g.~\cite[p.44]{gfology2}) that the generating function of the Catalan numbers $C_{l1}$ is given by\footnote{This generating function, touted as one of the most celebrated generating functions in combinatorics, is typically expressed as $\frac{1-\sqrt{1-4x}}{2x}$, which is easily converted to (\ref{eqn-genfun}).}
\begin{equation}\label{eqn-genfun}
f_1(x)=\frac{2}{1+\sqrt{1-4x}}
\end{equation}
Remember that a generating function of a sequence of number is a function such that the coefficients of its power series is exactly those of the sequence of numbers. Thus,
\begin{equation}
f_1(x)=C_{11}+C_{21}x+C_{31}x^2+\hdots=\sum_{l=1}^\infty C_{l1}x^{l-1}
\end{equation}
Therefore if we evaluate the function at $x=1/4$ we find that
\begin{equation}
f_1(\frac{1}{4})=\sum_{l=1}^\infty C_{l1}\frac{1}{4^{l-1}},
\end{equation}
thus we obtain
\begin{equation}
\sum_{l=1}^\infty   \frac{C_{l1}}{4^{l}}
=\frac{1}{4} \sum_{l=1}^\infty   \frac{C_{lj}}{4^{l-1}}=\frac{f_1(\frac{1}{4})}{4}=\frac{1}{2}.
\end{equation}
Following this argument, it is clear that
\begin{equation}
\sum_{l=j}^\infty   \frac{C_{lj}}{4^{l}}
=\frac{1}{4} \sum_{l=j}^\infty   \frac{C_{lj}}{4^{l-1}}=\frac{f_j(\frac{1}{4})}{4},
\end{equation}
where we define the sequence of generating functions $f_j$ as
\begin{equation}
f_j(x)=\sum_{l=j}^\infty C_{lj}x^{l-1}.
\end{equation}

Now since $C_{l2}=C_{l1}$ but obviously $C_{12}=0$, it is clear that $f_2=f_1-C_{11}=f_1-1$. Thus $f_2(1/4)=1$ and we find
\begin{equation}
\sum_{l=2}^\infty   \frac{C_{l2}}{4^{l}}
\frac{f_2(\frac{1}{4})}{4}=\frac{1}{4}.
\end{equation}

To find successive generating functions we use the properties of the Catalan's Triangle and make the following claim:
\begin{lemma}\label{lem-genfun}
For $n>1$
\begin{eqnarray}
f_n(x)=f_{n-1}(x)-xf_{n-2}(x)
\end{eqnarray}
\end{lemma}
\begin{proof}
To prove it, we write the definition of the generating function $f_{n-1}$
\begin{equation}
f_{n-1}(x)=\sum_{l=n-1}^\infty C_{l(n-1)}x^{l-1},
\end{equation}
and using the properties of the numbers
\begin{eqnarray}
f_{n-1}(x)&=&\sum_{l=n-1}^\infty C_{ln}x^{l-1}+\sum_{l=n-1}^\infty C_{(l-1)(n-2)}x^{l-1}
\nonumber  \\
&=&\sum_{l=n}^\infty C_{ln}x^{l-1}+x\sum_{l=n-2}^\infty C_{l(n-2)}x^{l-1}
\nonumber  \\
&=&
f_n(x)+xf_{n-1}x,
\end{eqnarray}
thus proving the claim.\qed
\end{proof}

Based on this fact, we can now prove another result.
\begin{lemma}
For $j\geq 1$, there holds
\begin{eqnarray}
 \sum_{l=j}^\infty   \frac{C_{lj}}{4^{l}}=\frac{1}{2^j}
\end{eqnarray}
\end{lemma}
\begin{proof}
We prove the Lemma by induction. It is already proved for $j=1,2$. Assume it for $j-1$ and $j-2$ and prove it for $j$. As we have shown
\begin{equation}
\sum_{l=j}^\infty   \frac{C_{lj}}{4^{l}}
=\frac{f_j(\frac{1}{4})}{4}
\end{equation}
But on the other hand $f_j(1/4)=f_{j-1}(1/4)-\frac{f_{j-2}(1/4)}{4}$, and substituting
\begin{eqnarray}
\sum_{l=j}^\infty   \frac{C_{lj}}{4^{l}}
&=&\frac{f_{j-1}(1/4)-\frac{f_{j-2}(1/4)}{4}}{4}
\nonumber \\ &=&
\sum_{l=j-1}^\infty   \frac{C_{lj}}{4^{l}}-\frac{1}{4} \sum_{l=j-2}^\infty   \frac{C_{lj}}{4^{l}}
\nonumber \\ &=&\frac{1}{2^{j-1}}-\frac{1}{4}\frac{1}{2^{j-2}}
\nonumber \\ &=&\frac{1}{2^{j}},
\end{eqnarray}
thus proving the Lemma.\qed
\end{proof}

Thus we finally obtain some partial sums in (\ref{eqn-Gsum}) as follows
 \begin{eqnarray}
 \vert G\vert &\leq& \sqrt{\bar \lambda} (\alpha-\beta) \frac{\mathrm{I_1}\left[2\sqrt{\bar \lambda \alpha \beta}\right]}{\sqrt{\alpha \beta}}+\sum_{i=0}^{\infty}  \sum_{j=1}^{j=\infty}\frac{F_{ij}}{2^{j}} 
 \nonumber \\ &=&
\sqrt{\bar \lambda} (\alpha-\beta) \frac{\mathrm{I_1}\left[2\sqrt{\bar \lambda \alpha \beta}\right]}{\sqrt{\alpha \beta}}
 \nonumber \\ &&
 +\sum_{i=0}^{\infty}  \sum_{j=1}^{j=\infty} 
 \frac{\bar\lambda^{i+1}\alpha^i\beta^i}{i!(i+1)!}(\alpha-\beta)\frac{\log^j\left(\frac{\alpha+\beta}{\alpha-\beta}\right)}{2^{j}j!},
 \end{eqnarray}
 which is a summable series both in $i$ and in $j$. Summing first in $i$ we find the same term as in (\ref{eqn-besselsum}), thus
  \begin{eqnarray}
 \vert G\vert &\leq&
\sqrt{\bar \lambda} (\alpha-\beta) \frac{\mathrm{I_1}\left[2\sqrt{\bar \lambda \alpha \beta}\right]}{\sqrt{\alpha \beta}}\left( \sum_{j=0}^{j=\infty} \frac{\log^j\left(\frac{\alpha+\beta}{\alpha-\beta}\right)}{2^{j}j!}\right),\quad
 \end{eqnarray}
 and the second term is the series of an exponential, therefore we finally reach
  \begin{eqnarray}
 \vert G\vert &\leq&
\sqrt{\bar \lambda} (\alpha-\beta) \frac{\mathrm{I_1}\left[2\sqrt{\bar \lambda \alpha \beta}\right]}{\sqrt{\alpha \beta}}\mathrm{e}^{\log \left(\sqrt{\frac{\alpha+\beta}{\alpha-\beta}}\right)}\nonumber \\ &=&
\sqrt{\bar \lambda(\alpha^2-\beta^2)} \frac{\mathrm{I_1}\left[2\sqrt{\bar \lambda \alpha \beta}\right]}{\sqrt{\alpha \beta}}.
 \end{eqnarray}
Substituting $\alpha$ and $\beta$ by the physical variables $r$, $\rho$, it is found that
   \begin{eqnarray}
 \vert G\vert &\leq&
\sqrt{\bar \lambda r\rho} \frac{\mathrm{I_1}\left[2\sqrt{\bar \lambda (r^2-\rho^2)}\right]}{\sqrt{ r^2-\rho^2}},
 \end{eqnarray}
 and going back to the original kernel $K$, we finally found that the successive approximation series converges and defines a kernel satisfying the following bound
 \begin{eqnarray}
 \vert K(r,\rho)\vert &\leq&
\rho\sqrt{\bar \lambda}  \frac{\mathrm{I_1}\left[2\sqrt{\bar \lambda (r^2-\rho^2)}\right]}{\sqrt{ r^2-\rho^2}},
 \end{eqnarray}
thus completing the proof of Theorem~\ref{th-mainsym}.

\section{Conclusion}\label{sec-conclusions}
This paper is a first step towards extending boundary stabilization results for constant-coefficient reaction-diffusion equations in disks to radially-varying coefficients. An assumption of revolution symmetry conditions has been made to simplify the equations, which become singular in the radius, complicating the design. The traditional backstepping method can be applied but the well-posedness of the kernel equation becomes challenging to prove. In this paper, a method of proof based on the properties of the  Catalan numbers has been successfully applied.

There are many open problems that still need to be tackled. For instance, the numerical solution of the kernel equation is not simple given the singularities that appear. Further regularity of the kernel is necessary to develop output-feedback results and does not seem to be simple to obtain. Extending the problem to spheres under revolution symmetry conditions is interesting from the point of view of applications, since these simplifications can be found in the engineering literature. Finally, dropping the revolution symmetry conditions would make the problem truly 2-D, but unfortunately the method of proof used in this work does not seem to extend, at least in a simple way.

%
%
\end{document}